\documentclass[11pt,a4paper,reqno]{amsart}
\usepackage{newtxtext}

\usepackage{epsfig}
\usepackage{amsfonts}
\usepackage{amsmath}
\usepackage{amssymb}
\usepackage{amsthm}
\usepackage{times}
\usepackage{xcolor}
\usepackage{enumerate}
\usepackage[T1]{fontenc}
\usepackage[utf8]{inputenc}
\usepackage{graphicx}
\usepackage{nicefrac}
\usepackage[includehead,includefoot,margin=25mm]{geometry}
\usepackage[all]{xy}
\usepackage{multicol}
\usepackage[colorlinks=true,linkcolor=blue,citecolor=blue,urlcolor=blue]{hyperref}
\usepackage{cleveref}
\usepackage{comment}

\newtheorem{theorem}{Theorem}[section]
\newtheorem*{theorem*}{Theorem}
\newtheorem{corollary}[theorem]{Corollary}
\newtheorem{lemma}[theorem]{Lemma}
\newtheorem{proposition}[theorem]{Proposition}
\newtheorem{question}[theorem]{Question}

{\theoremstyle{remark}
  \newtheorem{remark}[theorem]{Remark}}
{\theoremstyle{definition}

  \newtheorem{example}[theorem]{Example}

}

\newcommand{\LND}[0]{\ensuremath{\operatorname{LND}}}
\newcommand{\GM}[0]{\ensuremath{\mathbb{G}_{\mathrm{m}}}}
\newcommand{\Aut}[0]{\ensuremath{\operatorname{Aut}}}

\makeatletter \newcommand*\bigcdot{\mathpalette\bigcdot@{.5}} \newcommand*\bigcdot@[2]{\mathbin{\vcenter{\hbox{\scalebox{#2}{$\m@th#1\bullet$}}}}} \makeatother

\begin{document}

\title{On isotropy group of  locally finite derivations on $\mathbb{K}[X,Y]$}

\author{Luis Cid}
\address{Instituto de Matem\'atica y F\'\i sica, Universidad de Talca,
  Casilla 721, Talca, Chile.}%
\email{luis.cid@utalca.cl}

\author{Marcelo Veloso} 
\address{Departamento de Fisíca Estatística e Matemática, Universidade Federal de S\~ao Jo\~ao del-Rei, Brasil}
\email{veloso@ufsj.edu.br}

\date{\today}

\subjclass[2020]{13N15, 14R10, 13B10}

\keywords{locally finite derivation, isotropy group, exponential automorphism, polynomial automorphism, Jordan decomposition}

\begin{abstract}
In this paper, we study the isotropy groups of locally finite derivations of the polynomial ring $\mathbb{K}[X,Y]$, using Van den Essen's classification of locally finite derivations in two variables. We compare the isotropy group of a locally finite derivation with that of its associated exponential automorphism, showing that they coincide in the locally nilpotent case, whereas they may differ when the semisimple part is nontrivial. We also prove that every nonzero locally finite derivation has a nontrivial isotropy group.
\end{abstract}

\maketitle

\section{Introduction}

Let $\mathbb{K}$ be an algebraically closed field of characteristic zero, and let $B$ be an affine $\mathbb{K}$-algebra. We denote by $\mathbb{K}[X_1,\ldots,X_m]$ the polynomial ring in $m$ variables over $\mathbb{K}$, and in some cases we simply write $\mathbb{K}^{[m]}$. We write $\operatorname{End}(B)$ for the monoid of $\mathbb{K}$-algebra endomorphisms of $B$, and $\operatorname{Aut}(B)$ for the group of $\mathbb{K}$-algebra automorphisms of $B$.

A derivation on $B$ is a $\mathbb{K}$-linear map $D\colon B\to B$ satisfying the Leibniz rule
\[
D(ab)=aD(b)+bD(a), \qquad a,b\in B.
\]
We denote by $\operatorname{Der}(B)$ the set of all derivations of $B$. The group $\operatorname{Aut}(B)$ acts on $\operatorname{End}(B)$ by conjugation:
\[
\operatorname{Aut}(B)\times \operatorname{End}(B)\to \operatorname{End}(B),\qquad
(\varphi,E)\mapsto \varphi E\varphi^{-1}.
\]
For a derivation $D\in \operatorname{Der}(B)$, the stabilizer of $D$ under this action is called the \emph{isotropy group} of $D$, and is defined by
\[
\operatorname{Aut}(B)_D
:=
\{\varphi\in \operatorname{Aut}(B)\mid \varphi D\varphi^{-1}=D\}
=
\{\varphi\in \operatorname{Aut}(B)\mid \varphi D=D\varphi\}.
\]

If two derivations $D$ and $D'$ are conjugate, say $D'=\varphi D\varphi^{-1}$ for some $\varphi\in \operatorname{Aut}(B)$, then their isotropy groups are conjugate as subgroups of $\operatorname{Aut}(B)$:
\[
\operatorname{Aut}(B)_D=\varphi^{-1}\operatorname{Aut}(B)_{D'}\varphi.
\]
Therefore, the computation of isotropy groups may be reduced to representatives of conjugacy classes.

In recent years, several authors have investigated isotropy groups of derivations in various contexts. In \cite{mendes17plane} L. G. Mendes and I. Pan showed that if $D$ is a simple derivation on $\mathbb{K}^{[2]}$, then $\operatorname{Aut}_D(\mathbb{K}^{[2]})$ is trivial. Later,  L.~N.~Bertoncello and D.~Levcovitz, \cite{bertoncello2020isotropy}, proved that $\operatorname{Aut}_D(\mathbb{K}^{[m]})=\{id\}$ when $D$ is a simple Shamsuddin derivation. The converse statement was subsequently established by D.~Yan \cite{yan2024shamsuddin}. In another direction, R.~Baltazar and M.~Veloso \cite{BalVel21} studied isotropy groups of locally nilpotent derivations on certain Danielewski surfaces, namely surfaces defined by equations of the form
\[
f(X)Y-\varphi(Z),
\]
over an algebraically closed field of characteristic zero.

The main purpose of this paper is to study isotropy groups of locally finite derivations of the polynomial ring $\mathbb{K}[X,Y]$. Using Van den Essen's classification of locally finite derivations in two variables, we determine the isotropy group in each case and prove, in particular, that every nonzero locally finite derivation has a nontrivial isotropy group.

In the last section, we extend this point of view to automorphisms arising from the exponential of locally finite derivations. Since conjugation is compatible with exponentials,
\[
\varphi\exp(D)\varphi^{-1}=\exp(\varphi D\varphi^{-1}),
\]
it is natural to compare the isotropy group of a locally finite derivation with that of its associated exponential automorphism. We show that these two groups coincide in the locally nilpotent case, whereas they may differ when the semisimple part is nontrivial. Thus, besides isotropy groups of locally finite derivations, we also study isotropy groups of the corresponding exponential automorphisms on $\mathbb{K}[X,Y]$.

\smallskip
\noindent\textbf{Organization of the paper.}
Section~\ref{sec:prelim} recalls the necessary background: locally finite derivations and automorphisms, the Jordan--Chevalley decomposition, and the exponential map. Section~\ref{sec:isotropyLFD} contains the main computations: we determine the isotropy group of each normal form of locally finite derivation on $\mathbb{K}[X,Y]$ (Theorems~\ref{th3.2}--\ref{th3.8}), and we deduce that every nonzero locally finite derivation has a nontrivial isotropy group (Proposition~\ref{cor:nontrivial}). Section~\ref{sec:isotropyLFA} studies the isotropy groups of the corresponding exponential automorphisms, establishes when the isotropy of $D$ coincides with the isotropy of $\exp(D)$ (Propositions~\ref{prop:LND-equality} and~\ref{prop:criterion-equality}), and concludes with a complete description for locally finite automorphisms on $\mathbb{C}[X,Y]$ (Theorem~\ref{th:IsotropyLFA}).

\section{Preliminaries}\label{sec:prelim} 

This section is devoted to setting up notation and recalling the main concepts and results to be used throughout the paper.

\subsection{Locally finite endomorphisms and derivations}

An endomorphism $E \in \operatorname{End}(B)$ is called \emph{locally finite} if, 
for every $b \in B$, the $\mathbb{K}$-vector space generated by $\{E^{(n)}(b) \mid n \ge 0\}$
is finite-dimensional; for short we denote $E^{(n)}$ by $E^n$. Equivalently, for each $b \in B$, there exists a monic polynomial $p_b(T) \in \mathbb{K}[T]$ such that $p_b(E)(b)=0$, that is, the element $b$ satisfies a polynomial relation with respect to $E$. Analogously, we define a \emph{locally finite derivation} on $B$.  The set of locally finite derivations is denoted by $\operatorname{LFD}(B)$ and the set of locally finite automorphisms by $\operatorname{LFA}(B)$.

Let $D$ be a nonzero locally finite derivation on $\mathbb{K}[X]$. It is easy to verify that  $D$ must be of the form $D = (aX + b)\frac{\partial}{\partial X}$. In the polynomial ring in two variables, we have the following characterization of locally finite derivations, up to the action of an automorphism.

\begin{lemma}[Corollary 4.7, \cite{V92}]
\label{lfdclas}
Let $D\neq 0$ be a locally finite derivation on $\mathbb{K}[X,Y]$. Then there exists an automorphism $\varphi \in \operatorname{Aut}(\mathbb{K}[X,Y])$ such that $\varphi D \varphi^{-1}$ is one of the following:
\begin{enumerate}
    \item $D = f(X)\dfrac{\partial}{\partial Y}$, where $f(X) \in \mathbb{K}[X]$, $f\neq 0$;
    \item $D = \dfrac{\partial}{\partial X} + bY\dfrac{\partial}{\partial Y}$, where $b\in\mathbb{K}^*$;
    \item $D = aX\dfrac{\partial}{\partial X} + (amY + X^m)\dfrac{\partial}{\partial Y}$, with $a\in\mathbb{K}^*$, $m\in \mathbb{Z}$, $m\geq 1$;
    \item $D = (aX + bY)\dfrac{\partial}{\partial X} + (cX + dY)\dfrac{\partial}{\partial Y}$, where $a,b,c,d \in \mathbb{K}$.
\end{enumerate}
\end{lemma}

\begin{remark}
We note that the derivation $\tfrac{\partial}{\partial X}$ (i.e.\ Type~(2) with $b=0$) falls under Type~(1) with $f=1$ in the above classification. In Theorem~\ref{th3.3} below we treat both $b\neq 0$ and $b=0$ together for completeness, noting that the $b=0$ case is already covered by Theorem~\ref{th3.2} (with $f(X)=1$, $\alpha=1$, $\beta=0$, $\gamma=1$). The two subcases exhibit qualitatively different isotropy groups, which justifies the unified treatment.
\end{remark}

\subsection{Jordan decomposition and the exponential map}

A derivation $D \in \operatorname{Der}(B)$ is called \emph{semisimple} if there exists a basis of semi-invariants $\{b_l\}_{l \in I}$ such that $D(b_l) = \lambda_l b_l$ with $\lambda_l \in \mathbb{K}$. The set of semisimple derivations is denoted by $\operatorname{SSD}(B)$. The derivation $D$ is called \emph{locally nilpotent} if for every $f \in B$, there exists $j \in \mathbb{Z}_{\geq 0}$ such that $D^j(f) = 0$. The set of such derivations is denoted by $\operatorname{LND}(B)$. Both semisimple and locally nilpotent derivations are particular cases of locally finite derivations.

Given $D \in \operatorname{LFD}(B)$, we know that $D$ admits a Jordan decomposition $D = D_s + D_n$, where $D_s$ is semisimple, $D_n$ is locally nilpotent, and $[D_s, D_n] = 0$ (see~\cite[Proposition~1.3.13]{van2000polynomial}). Similarly, if $\varphi \in \operatorname{LFA}(B)$, then it admits a decomposition $\varphi = \varphi_s \circ \varphi_u$, where $\varphi_s$ is semisimple, $\varphi_u$ is unipotent, and they commute.

Following Maubach's Ph.D.\ thesis \cite{Maubach03}, if $D \in \operatorname{LFD}(B)$, one defines its exponential  \linebreak
\(\exp(D)\colon B \to B\) 
by the formal power series
$\displaystyle\exp(D)(b) \;=\; \sum_{j \geq 0} \frac{1}{j!} D^j(b)$.

Since $D$ is locally finite, for each $b \in B$ there exists 
$n = n(b)$ such that $\{D^j(b)\}_{j>n}$ are linearly dependent on $\{D^j(b)\}_{j \le n}$; 
hence the exponential series defining $\exp(D)(b)$ is effectively finite (it stabilizes after finitely many terms).
This defines a locally finite automorphism whose inverse is $\exp(-D)$.

More generally, for any affine domain $B$, if $\{b_l\}_l$ is a basis of semi-invariants for $D_s$ we have $b = \sum_l \beta_l b_l \in B$ where $D_n^{n_l+1}(b_l) = 0$ and $D_s(b_l) = \lambda_l b_l$, then
\[
\exp(D_s)(b_l) = e^{\lambda_l} b_l, \quad \text{and} \quad \exp(D) = \exp(D_n) \circ \exp(D_s).
\]
Thus the explicit automorphism is given by
\[
\exp(D)(b) = \sum_l e^{\lambda_l} \left[b_l + D_n(b_l) + \frac{1}{2!} D_n^2(b_l) + \cdots + \frac{1}{n_l!} D_n^{n_l}(b_l)\right].
\]

We consider the action of $\operatorname{Aut}(B)$ on $\operatorname{LFD}(B)$ by conjugation:
\[
\begin{aligned}
    \operatorname{Aut}(B)\times \operatorname{LFD}(B)&\longrightarrow \operatorname{LFD}(B) \\
    (\varphi,D) &\mapsto \varphi D \varphi^{-1}.
\end{aligned}
\]

This action is well defined because the minimal polynomial of $D$ and $\varphi D \varphi^{-1}$ are the same. Hence, for all $b\in B$ the $\mathbb{K}$-vector space $\{(\varphi D \varphi^{-1})^n(b) = \varphi D^n \varphi^{-1}(b)\}_{n \geq 0}$ is finite-dimensional.

The stabilizer of $D$ under this action is its isotropy group, $\operatorname{Aut}(B)_D$. Moreover, there is an equivalence between the isotropy groups of conjugate derivations: if $D' = \varphi D \varphi^{-1}$, then
\[
\operatorname{Aut}(B)_D = \varphi \operatorname{Aut}(B)_{D'} \varphi^{-1}.
\]
Thus, it suffices to compute isotropy groups up to conjugacy classes. The orbit of $D$ under this action is the conjugacy class
\[
\mathcal{O}(D) = \operatorname{Cl}(D) := \{\varphi D \varphi^{-1} \mid \varphi \in \operatorname{Aut}(B)\}.
\]
A conjugacy class is called \emph{closed} if it contains all its elements under conjugation.

In \cite{FurMa2010}, Furter and Maubach proved that the conjugacy class of any semisimple automorphism in $\mathbb{K}^2$ is closed. In particular, any such automorphism is conjugate to a diagonal automorphism of the form $\varphi \circ (aX, bY) \circ \varphi^{-1}$  where $a,b\in \mathbb{K}^{\ast}$. As a consequence, any semisimple derivation with integer eigenvalues in $\operatorname{Der}(\mathbb{K}[X,Y])$ is conjugate to a linear derivation of the form $\alpha X \dfrac{\partial}{\partial X} + \beta Y \dfrac{\partial}{\partial Y}$ where $\alpha, \beta\in \mathbb{K}$.

Regular $\GM$-actions are a particular case of semisimple automorphisms. The problem of linearization for such actions in dimension $n > 3$ remains open. Koras, Russell, and Makar-Limanov showed in \cite{KaKoLiRu} that $\GM$-actions on $\mathbb{A}^3$ are linearizable. These actions correspond to semisimple derivations with integer eigenvalues, although not all derivations arise in this way.

\section{Isotropy groups of locally finite derivations}\label{sec:isotropyLFD}

In this section, we determine the isotropy groups of locally finite derivations on polynomial ring in two variables. By Lemma~\ref{lfdclas}, the classification  reduces the analysis to four distinct cases.

We begin with a general result, independent of this classification, establishing that the isotropy group is always nontrivial. The subsequent explicit computations then provide a precise description of the isotropy group in each case.

\begin{proposition}\label{cor:nontrivial}
Let $\mathbb{K}$ be an algebraically closed field of characteristic zero and let $B$ be an affine $\mathbb{K}$-algebra. If $D\in \operatorname{LFD}(B)$ with $D\neq 0$, then $\operatorname{Aut}(B)_D$ is nontrivial.
\end{proposition}

\begin{proof}
For each $t\in\mathbb{K}$, the element $tD$ is again locally finite, and the automorphism $\exp(tD)$ is well-defined. The conjugation identity $\exp(tD)\,D\,\exp(-tD)=D$ (which follows from $[tD,D]=0$) shows that $\exp(tD)\in \operatorname{Aut}(B)_D$ for all $t\in\mathbb{K}$. Since $D\neq 0$, there exists $b\in B$ such that $D(b)\neq 0$. Fix such a $b$. The map
\[
t\mapsto \exp(tD)(b)=\sum_{j\ge 0}\frac{t^j}{j!}D^j(b)
\]
is a polynomial in $t$ over $\mathbb{K}$ whose coefficient of $t$ is $D(b)\neq 0$, hence it is nonconstant. Since $\mathbb{K}$ is infinite, there exists $t_0\in\mathbb{K}$ such that $\exp(t_0D)(b)\neq b$. This means $\exp(t_0D)\neq\mathrm{id}_B$. Observe that this conclusion concerns only the action on the specific element $b$, so no cancelation with other elements can restore the identity. Therefore, $\operatorname{Aut}(B)_D$ contains the nontrivial automorphism $\exp(t_0D)$.
\end{proof}

\begin{remark}
Since the isotropy group of every nonzero locally finite derivation on $\mathbb{K}^{[2]}$ is nontrivial, it follows from \cite{mendes17plane} that no simple derivation on $\mathbb{K}^{[2]}$ can be locally finite.
\end{remark}

The next theorem determines the isotropy group of the nilpotent derivations, item (1) of Lemma \ref{lfdclas}.

\begin{theorem}\label{th3.2}
	Let $D = f(X)\dfrac{\partial}{\partial Y}$ be a derivation of $\mathbb{K}[X]$, where $f(X) \in \mathbb{K}[X]$ and  $f(X)\neq 0$. Then
	\[
	\Aut_D(\mathbb{K}[X,Y])
	=\{(\alpha X+\beta,\; \gamma Y+p(X)) \mid \alpha,\gamma\in\mathbb{K}^*,\ \beta\in\mathbb{K},\ p\in\mathbb{K}[X],\
	f(\alpha X+\beta)=\gamma f(X)\}.
	\]
\end{theorem}

\begin{proof}
	Let $\rho \in \Aut_D(\mathbb{K}[X,Y])$. Write
	\[
	\rho(X)=\sum_{i=0}^t a_i(X)Y^i, 
	\qquad 
	\rho(Y)=\sum_{j=0}^s b_j(X)Y^j,
	\]
	with $a_i(X), b_j(X) \in \mathbb{K}[X]$. As $D(X)=0$ and $\rho D = D\rho$, we obtain
	\[
	0 = D(\rho(X)) = \sum_{i=1}^t i\,a_i(X)f(X)Y^{i-1}.
	\]
	Because $f(X)\neq 0$, it follows that $a_i(X)=0$ for all $i\geq 1$, hence $\rho(X)=q(X)\in\mathbb{K}[X]$.
	
	Next, from $\rho(D(Y)) = D(\rho(Y))$, we have
	\[
	f(\rho(X)) = D(\rho(Y)) = \sum_{j=1}^s j\,b_j(X)f(X)Y^{j-1}.
	\]
	As the left-hand side is independent of $Y$, we conclude that $b_j(X)=0$ for all $j\geq 2$. Thus,
	\[
	\rho(Y)=p(X)+r(X)Y, \mbox{ with } p(X), r(X)\in\mathbb{K}[X].
	\]
	
	Since $\rho$ is an automorphism, its Jacobian determinant satisfies
	\[
	\det J(\rho) =
	\begin{vmatrix}
		q'(X) & 0 \\
		p'(X)+r'(X)Y & r(X)
	\end{vmatrix}
	= q'(X)\,r(X) \in \mathbb{K}^*.
	\]
	Hence $q'(X), r(X) \in \mathbb{K}^*$, and therefore $q(X)=\alpha X+\beta$ and $r(X)=\gamma$ with $\alpha,\gamma\in\mathbb{K}^*$.
	
	Finally, from $\rho(D(Y))=D(\rho(Y))$ we obtain
	\[
	f(\alpha X+\beta)=\gamma f(X).
	\]
	
	Thus,
	\[
	\rho(X)=\alpha X+\beta, \qquad \rho(Y)=\gamma Y+p(X),
	\]
	with $p(X)\in\mathbb{K}[X]$, $\alpha,\gamma\in\mathbb{K}^*$ and $\beta\in\mathbb{K}$ satisfying $f(\alpha X+\beta)=\gamma f(X)$.
\end{proof}

The next result describes the isotropy group of the derivations in item (2) of Lemma \ref{lfdclas}.

\begin{theorem}\label{th3.3}
Let $D = \frac{\partial}{\partial X} + bY \frac{\partial}{\partial Y}$ be a derivation of $\mathbb{K}[X,Y]$, where $b \in \mathbb{K}^*$.
Then
    \[
    \operatorname{Aut}_D(\mathbb{K}[X,Y]) = \left\{ (X + \beta,\; \gamma Y) \;\middle|\; \beta \in \mathbb{K},\ \gamma\in \mathbb{K}^* \right\}.
    \]
 \end{theorem}

\begin{proof}
Let $\rho \in \operatorname{Aut}(\mathbb{K}[X,Y])$ be given by
\[
\rho(X) = \sum_{i=0}^t a_i(X) Y^i, \qquad \rho(Y) = \sum_{j=0}^s b_j(X) Y^j.
\]
Suppose that $\rho D = D \rho$. We first compute
\[
D(\rho(X)) = \sum_{i=0}^t a_i'(X) Y^i + \sum_{i=1}^t i b a_i(X) Y^i = a_0'(X) + \sum_{i=1}^t \left( a_i'(X) + i b a_i(X) \right) Y^i.
\]
Equating this with $\rho(D(X)) = 1$, we get:
\[
\begin{cases}
a_0'(X) = 1, \\
a_i'(X) + i b a_i(X) = 0 \quad \text{for all } i \geq 1.
\end{cases}
\]
The first equation implies $a_0(X) = X + \alpha, \quad \alpha \in \mathbb{K}$. For the second equation the only polynomial solution is $a_i(X)=0$, because   $b \neq 0$. Thus

\[
\rho(X) = X + \alpha,\ \alpha \in \mathbb{K}, \quad \text{if }   b \neq 0.
\]
Now compute
\[
D(\rho(Y)) = b_0'(X) + \sum_{j=1}^s \left( b_j'(X) + j b b_j(X) \right) Y^j
\mbox{ and }
\rho(D(Y)) = b \rho(Y) = \sum_{j=0}^s b\, b_j(X) Y^j.
\]
Comparing both expressions gives
\[
\begin{cases}
b_0'(X) = b\, b_0(X), \\
b_j'(X) + j b b_j(X) = b\, b_j(X) \quad \text{for } j\geq 1.
\end{cases}
\]
The first equation has no nonzero polynomial solution (degree argument) and $b\neq 0$. So $b_0(X) = 0$. The second equation implies    $b_1'(X)=0$, if $j=1$. Hence $b_1(X)\in\mathbb{K}^*$. If $j\geq 2$, the same degree argument shows $b_j(X)=0$. Thus
$\rho(Y) = \gamma Y$ with  $\gamma \in \mathbb{K}^*$.

\end{proof}

Observe that, if $b = 0$ in Theorem~\ref{th3.3}, then $D = \frac{\partial}{\partial X}$. This case is already covered by Theorem~\ref{th3.2} by taking $f(X) = 1$ and performing the change of variables $(X,Y) \mapsto (Y,X)$. In this setting, the condition $f(\alpha X + \beta) = \gamma f(X)$ forces $\gamma = 1$. Therefore,
    \[
    \operatorname{Aut}_D(\mathbb{K}[X,Y]) = \left\{ (X + r(Y),\; \alpha Y + \beta) \;\middle|\; \alpha\in \mathbb{K}^*,\ \beta \in \mathbb{K},\ r(Y) \in \mathbb{K}[Y] \right\}.
    \]
   
The next result describes the isotropy group of the derivations in item (3) of Lemma \ref{lfdclas}.
    \begin{theorem}
\label{th3.4}
		Let $D=aX\dfrac{\partial}{\partial X}+(amY+X^m)\dfrac{\partial}{\partial Y}$ be a derivation of $\mathbb{K}[X,Y]$,	where   $a \in \mathbb{K}^{*}$ and $m\geq 2$. Then 
\[ 
\operatorname{Aut}_D(\mathbb{K}[X,Y])=
\{(\alpha^m X, \alpha Y+\beta X^m)  \mid \alpha \in \mathbb{K}^*,\ \beta \in \mathbb{K}\}.
\]
\end{theorem}
\begin{proof}
The proof follows the same strategy as in the previous theorem. Let $\rho \in \Aut_D(\mathbb{K}[X,Y])$.  Using the commutativity condition $\rho D = D\rho$, together with a comparison of degrees in $Y$, one shows that $\rho(X) = cX$ and $\rho(Y) = p(X) + \alpha Y$, with $c,\alpha \in \mathbb{K}^*$ and $p(X)\in \mathbb{K}[X]$. Substituting into the relation $D(\rho(Y)) = \rho(D(Y))$ and comparing degrees in $X$, it follows that $p(X) = \beta X^m$ and $\alpha = c^m$.

Finally, the Jacobian condition ensures that $c \in \mathbb{K}^*$, and hence
\[
\rho = (cX,\, c^m Y + \beta X^m).
\]
Reparametrizing, we obtain the desired description.
\end{proof}

\subsection{The linear case}

For a linear derivation of the form 
\[
D = (aX + bY)\frac{\partial}{\partial X} + (cX + dY)\frac{\partial}{\partial Y},
\]
we reduce our analysis to two cases, based on the Jordan decomposition of the associated matrix:

\[
D_1 = aX \frac{\partial}{\partial X} + bY \frac{\partial}{\partial Y} \quad \text{and} \quad D_2 = (aX + Y)\frac{\partial}{\partial X} + aY \frac{\partial}{\partial Y}.
\]

\begin{remark}
\label{obslinear}
Let $M \in M_2(\mathbb{K})$ be the matrix associated with a linear derivation in two variables. By the Jordan decomposition theorem, $M$ is similar to a Jordan matrix $J$, that is, there exists an invertible matrix $P \in \operatorname{GL}_2(\mathbb{K})$ such that
\[
M = PJP^{-1}, \quad \text{where } J = 
\begin{bmatrix}
\lambda_1 & 0 \\
0 & \lambda_2
\end{bmatrix}
\quad \text{or} \quad
\begin{bmatrix}
\lambda & 1 \\
0 & \lambda
\end{bmatrix}, \mbox{ for some } \lambda_1, \lambda_2, \lambda \in \mathbb{K}.
\]

Consequently, every linear derivation is conjugate to one of the following canonical forms:
\[
D_1 = \lambda_1 X \frac{\partial}{\partial X} + \lambda_2 Y \frac{\partial}{\partial Y} \quad \text{or} \quad D_2 = (\lambda X + Y)\frac{\partial}{\partial X} + \lambda Y \frac{\partial}{\partial Y}.
\]

In particular, if the linear derivation $D_M$ corresponds to matrix $M$ and $M=PJP^{-1}$, then
\[
\operatorname{Aut}_{D_M}(\mathbb{K}[X, Y]) = P^{-1} \operatorname{Aut}_{D_J}(\mathbb{K}[X, Y]) P,
\]
where $P \in \operatorname{GL}_2(\mathbb{K})$ acts on $\operatorname{Aut}(\mathbb{K}[X,Y])$ by the corresponding linear change of variables.
\end{remark}

	\begin{theorem} \label{th3.6}
Let $D = aX \dfrac{\partial}{\partial X} + aY \dfrac{\partial}{\partial Y}$ be a linear derivation of \, $\mathbb{K}[X,Y]$, where   $ a\in \mathbb{K}^{\ast}$. Then 
\[
\operatorname{Aut}(\mathbb{K}[X,Y])_D=\{ (\alpha X+\beta Y,\gamma X+\delta Y) \mid  \alpha, \beta,\gamma, \delta\in \mathbb{K},\ \alpha\delta-\beta\gamma\neq 0\}.
\]    
\end{theorem}

\begin{proof}
Let $\rho(X)=\displaystyle\sum^t_{i=0}a_i(X) Y^i$ and $\rho(Y)=\displaystyle\sum^{s}_{j=0}b_j(X)Y^j$. Suppose that $\rho D=D\rho$. Applying $D$ to $\rho(X)$, we obtain
\[
D(\rho(X)) = D\left( \sum_{i=0}^t a_i(X) Y^i \right) = \sum_{i=0}^t  aa_i^{'}(X)XY^i + \sum_{i=1}^t a i a_i(X)Y^{i}.
\]
On the other hand,
\[
\rho(D(X)) = \rho(aX) = a \rho(X) = \sum_{i=0}^t aa_i(X) Y^i.
\]
Equating coefficients in $D(\rho(X)) = \rho(D(X))$, we get:
\[
a_0(X)=a_0^{\prime}(X)X  \quad \mbox{ and } \quad a_i^{\prime}(X)X = (1 - i) a_i(X), \mbox{ for } i\geq 1.
\]
The general solution of $Xa_i'(X)=(1-i)a_i(X)$ is $a_i(X) = \lambda_i X^{1-i}$. For this to be a polynomial we need $1-i \geq 0$, i.e.\ $i \leq 1$. Thus $a_0(X)=\alpha$, $a_1(X)=\beta \in \mathbb{K}$ and $a_i(X) =0$ for $i\geq 2$. Thus
\[
\rho(X)=\alpha X + \beta Y, \mbox{ where } \alpha, \beta \in \mathbb{K}.
\]
Similarly, we determine that
\[
\rho(Y)=\gamma X + \delta Y, \mbox{ where } \gamma, \delta \in \mathbb{K}.
\]
Since $\rho\in\Aut(\mathbb{K}[X,Y])$, its Jacobian determinant equals
$\det\begin{pmatrix}\alpha & \beta\\ \gamma & \delta\end{pmatrix}
=\alpha\delta-\beta\gamma\in\mathbb{K}^{\ast}$.
\end{proof}

\begin{theorem} \label{th3.7}
\label{tlinear}
		Let $D = aX \dfrac{\partial}{\partial X} + bY \dfrac{\partial}{\partial Y}$ be a linear derivation of \, $\mathbb{K}[X,Y]$, where   $a,\, b\in \mathbb{K}^*$ and $a \neq b$. Then 
    \[
      \operatorname{Aut}(\mathbb{K}[X,Y])_D=\{(\alpha X,\beta Y) \mid \alpha,\beta\in \mathbb{K}^{\ast}\}.
    \]
\end{theorem}

\begin{proof}
The proof uses the same arguments as Theorem~\ref{th3.6}. Since $D(X)=aX$ and $D(Y)=bY$ with $a\neq b$, the equation $Xa_i'(X)=(1-i)a_i(X)$ for $\rho(X)=\sum_i a_i(X)Y^i$ forces $a_1(X)=\beta\in\mathbb{K}$, but the condition $\rho(D(X))=D(\rho(X))$ at the $Y^1$ term gives $a\beta Y = b\beta Y$, hence $\beta=0$ (since $a\neq b$). Thus $\rho(X)=\alpha X$, and similarly $\rho(Y)=\delta Y$. The Jacobian condition forces $\alpha,\delta\in\mathbb{K}^*$.
\end{proof}

\begin{theorem} \label{th3.8}
		Let $D=(aX +Y)\dfrac{\partial}{\partial X}+aY\dfrac{\partial}{\partial Y}$ be a nonzero linear derivation of $\mathbb{K}[X,Y]$, where   $a \in \mathbb{K}^*$. Then 
        \[
        \operatorname{Aut}(\mathbb{K}[X,Y])_D= \{(\alpha X +\beta Y,\,\alpha Y ) \mid \alpha \in \mathbb{K}^*,\ \beta \in \mathbb{K}\}.
        \]
\end{theorem}

\begin{proof}
Assume that $\rho$ is an automorphism on $\mathbb{K}[X,Y]$ defined by
\[
\rho(X) = \sum_{i=0}^t a_i(X) Y^i, \qquad \rho(Y) = \sum_{j=0}^s b_j(X) Y^j,
\]
where  $a_i(X), b_j(X) \in \mathbb{K}[X]$. We want to determine when $\rho D = D \rho$. By definition, we have
\[
\rho(D(Y)) = a \rho(Y) = a \sum_{j=0}^s b_j(X) Y^j
\]
and
\begin{align*}
D(\rho(Y)) 
&= D\left( \sum_{j=0}^s b_j(X) Y^j \right)
=\sum_{j=0}^sab_j^{'}(X)XY^j+\sum_{j=0}^sb_j^{'}(X)Y^{j+1} + \sum_{j=1}^sjab_j(X) Y^j.
\end{align*}
Since $\rho(D(Y)) = D(\rho(Y))$, equating coefficients of each power of $Y$ yields:

\begin{equation*}
    \begin{cases} 
    b_0(X) = b_0'(X) X,\\
    b_s'(X) = 0, \\ 
   a b_j(X) =  a b_j'(X) X + b_{j-1}'(X) + ajb_j(X), \quad 0<j<s.\\
    \end{cases}
    \end{equation*}

The first equation $b_0(X)=Xb_0'(X)$ has general polynomial solution $b_0(X)=cX$ for some $c\in\mathbb{K}$. Substituting into the equation for $j=1$ (with $b_0'(X)=c$), we obtain  
\[
ab_1(X) = ab_1'(X)X + b_0'(X) + ab_1(X). 
\]
This implies  $0 = ab_1'(X)X + c$. Then $c=0$ and $b_1'(X)=0$, giving $b_0(X)=0$ and $b_1(X)=\alpha\in\mathbb{K}$. For $j\geq 2$, the equation $a(1-j)b_j(X) = ab_j'(X)X + b_{j-1}'(X)$ with $b_{j-1}'=0$ (established inductively) gives $a(1-j)b_j(X)=ab_j'(X)X$, i.e.\ $Xb_j'(X)=(1-j)b_j(X)$, whose only polynomial solution for $j\geq 2$ is $b_j(X)=0$. Thus
\[
\rho(Y) = \alpha Y, \quad \alpha \in \mathbb{K}^*.
\]

Now compute 
\[
\rho(D(X)) = \rho(aX + Y) = a \rho(X) + \alpha Y=\sum_{i=0}^t aa_i(X) Y^i+\alpha Y.
\]
On the other hand,
\begin{align*}
D(\rho(X)) 
&= \sum_{i=0}^t aa_i^{'}(X)XY^i+\sum_{i=0}^t a_i^{'}(X)Y^{i+1} + \sum_{i=1}^t iaa_i(X)Y^i.
\end{align*}
From $\rho(D(X))=D(\rho(X))$, equating coefficients for each power of $Y$ gives the system:
\begin{equation*} 
    \begin{cases}
    a_s'(X) = 0,\\
    a_0(X) = a_0'(X) X,\\
    \alpha = a a_1'(X) X + a_0'(X), \\ 
   a(1-i) a_i(X) = a a_i'(X) X + a_{i-1}'(X), \quad 1<i\leq s.\\
    \end{cases}
    \end{equation*}
Solving this system gives $a_0(X)=\gamma X$ and  $\alpha = a\cdot 0 + \gamma = \gamma$. So $a_0(X)=\alpha X$. From the equations for $i\geq 2$ one finds $a_i(X)=0$, and the  equation for $i=1$ we obtain $a_1(X)=\beta\in\mathbb{K}$. Therefore, 
\[
\rho=(\alpha X + \beta, \, \alpha Y ), \mbox{ where } \quad \alpha,\beta \in \mathbb{K} \mbox{ and }\ \alpha\neq 0.
\]
\end{proof}

In particular, Proposition ~\ref{cor:nontrivial} applies to each of the four normal forms of Lemma~\ref{lfdclas}, confirming that no nonzero locally finite derivation on $\mathbb{K}[X,Y]$ has a trivial isotropy group.

\section{Isotropy groups for exponential automorphisms}\label{sec:isotropyLFA}
 
In this section we specialize to $\mathbb{K}=\mathbb{C}$. This restriction is necessary because the exponential map $\exp(D)$ is defined via the complex exponential function $e^\lambda$, and key injectivity properties (used in Propositions~\ref{prop:LND-equality} and~\ref{prop:criterion-equality}) rely on analytic properties of $e^\lambda$ over $\mathbb{C}$. All results of Section~\ref{sec:isotropyLFD} apply in particular over $\mathbb{C}$.
 
\subsection{The exponential map for locally finite derivations}
 
Let $D\in \operatorname{LFD}(B)$ and let $b\in B$. Consider the $\mathbb{C}$-vector subspace
$V_b:=\operatorname{span}_{\mathbb{C}}\{D^n(b)\mid n\geq 0\}\subset B$.
Since $D$ is locally finite, the space $V_b$ is finite-dimensional. Furthermore, $V_b$ is stable under $D$, because $D(D^n(b))=D^{n+1}(b)\in V_b$ for all $n\geq 0$.
Hence the restriction $D_b:=D|_{V_b}\in \operatorname{End}_{\mathbb{C}}(V_b)$
is a well-defined linear operator on the finite-dimensional vector space $V_b$.
 
We may therefore define
$\exp(D_b):=\displaystyle\sum_{n\geq 0}\frac{D_b^n}{n!}\in \operatorname{End}_{\mathbb{C}}(V_b)$,
and then set $\exp(D)(b):=\exp(D_b)(b)\in V_b\subset B$.
In this way, we obtain a well-defined map $\exp(D)\colon B\to B$. It is standard that $\exp(D)$ is a $\mathbb{C}$-algebra automorphism, and that its inverse is $\exp(-D)$ (see~\cite{Maubach03}).
 
Moreover, for every $\varphi\in \operatorname{Aut}(B)$, the following identity holds:
\begin{equation}\label{eq:conj-exp}
\varphi\,\exp(D)\,\varphi^{-1}=\exp(\varphi D\varphi^{-1}).
\end{equation}
Consequently, we obtain the commutative diagram
\begin{multicols}{2}
$
\begin{array}[c]{ccc}
\nonumber \operatorname{Aut}(B)\times \operatorname{LFD}(B)&\xrightarrow{} &\operatorname{LFD}(B)\\
\nonumber \Big\downarrow\scriptstyle{\operatorname{id}\times\exp(-)} &\circlearrowleft&\Big\downarrow\scriptstyle{\exp(-)}\\
\nonumber \operatorname{Aut}(B)\times \operatorname{LFA}(B)&\xrightarrow{}&\operatorname{LFA}(B)
\end{array}
\begin{array}[c]{ccc}
\nonumber (\varphi, D)&\to &\varphi D\varphi^{-1}\\
\nonumber \Big\downarrow &\circlearrowleft &\Big\downarrow\\
\nonumber (\varphi,\exp(D))&\to &\varphi\exp(D)\varphi^{-1}=\exp(\varphi D\varphi^{-1})
\end{array}
$
\end{multicols}
 
\subsection{Isotropy of $D$ and of $\exp(D)$}
 
For $D\in \operatorname{LFD}(B)$, define 
\[
\Aut(B)_D=\{\varphi\in\Aut(B)\mid \varphi D\varphi^{-1}=D\} \mbox{ and } 
\]
\[
\Aut(B)_{\exp(D)}=\{\varphi\in\Aut(B)\mid \varphi\exp(D)\varphi^{-1}=\exp(D)\}
\]
By~\eqref{eq:conj-exp}, we always have $\Aut(B)D\subseteq \Aut(B){\exp(D)}$. In general, equality need not hold without additional hypotheses, since the exponential map is not injective on semisimple parts over $\mathbb{C}$.
 
\begin{lemma}\label{lem:comm-exp}
Let $D,D'\in\operatorname{LFD}(B)$. If $[D',D]=0$, then $\exp(D')\in\Aut(B)_D$.
Conversely, if $D'\in\LND(B)$ and $\exp(D')\in\Aut(B)_D$, then $[D',D]=0$.
\end{lemma}
 
\begin{proof}
If $[D',D]=0$, then $\operatorname{ad}_{D'}(D)=0$, and hence
$
\exp(D')\,D\,\exp(-D')
=
\exp(\operatorname{ad}_{D'})(D)
=
D$, so $\exp(D')\in\Aut(B)_D$.
 
Conversely, assume $D'\in\LND(B)$ and $\exp(D')D\exp(-D')=D$. Then $\exp(\operatorname{ad}_{D'})(D)=D$.
Since $D'$ is locally nilpotent, $\operatorname{ad}_{D'}$ is locally nilpotent on $\operatorname{Der}_{\mathbb{C}}(B)$; in particular it is nilpotent on the finite-dimensional space $W=\operatorname{Span}_{\mathbb{C}}\{\operatorname{ad}_{D'}^{n}(D)\mid n\ge 0\}$.
Hence $\exp(\operatorname{ad}_{D'})=\sum_{j=0}^{N}\frac{1}{j!}\operatorname{ad}_{D'}^{j}$ on $W$. From $\exp(\operatorname{ad}_{D'})(D)=D$ we get $\sum_{j\ge 1}\frac{1}{j!}\operatorname{ad}_{D'}^{j}(D)=0$.
Factoring $\operatorname{ad}_{D'}$, this becomes
$\operatorname{ad}_{D'}\!\left(\sum_{j\ge 0}\frac{1}{(j+1)!}\operatorname{ad}_{D'}^{j}(D)\right)=0$.
The operator $Q(N):=\sum_{j\ge 0}\frac{N^j}{(j+1)!}$ (where $N=\operatorname{ad}_{D'}|_W$) satisfies $Q(0)=1$ and $N$ is nilpotent on $W$, so $Q(N)$ is invertible on $W$ by the Neumann series. We conclude $N(D)=\operatorname{ad}_{D'}(D)=0$, i.e.\ $[D',D]=0$.
\end{proof}
 
\begin{example}
For $D=X\frac{\partial}{\partial X}$ we have
$\big[\frac{\partial}{\partial Y},\,X\frac{\partial}{\partial X}\big]=0$.
Thus $\exp\!\big(\frac{\partial}{\partial Y}\big)=(X,Y+1)\in \Aut(\mathbb{C}[X,Y])_D$.
\end{example}
 
\begin{example}
Consider the derivation $D=\tfrac{\partial}{\partial X}\in \operatorname{Der}(\mathbb{C}[X,Y])$.  
Its isotropy group is
\[
\operatorname{Aut}_{\frac{\partial}{\partial X}}(\mathbb{C}[X,Y])
=\left\{\,(X+p(Y),\,aY+b)\;\middle|\; p(Y)\in \mathbb{C}[Y],\; a\in \mathbb{C}^*,\; b\in \mathbb{C}\,\right\}.
\]
Indeed, if $\rho\in\Aut(\mathbb{C}[X,Y])$ satisfies $\rho\dfrac{\partial}{\partial X}=\dfrac{\partial}{\partial X}\rho$, then
$\frac{\partial}{\partial X}(\rho(X))=1$ forces $\rho(X)=X+p(Y)$ with $p\in\mathbb{C}[Y]$, and
$\frac{\partial}{\partial X}(\rho(Y))=0$ forces $\rho(Y)\in\mathbb{C}[Y]$. The Jacobian condition gives
$\rho(Y)=aY+b$ with $a\in\mathbb{C}^*$, $b\in\mathbb{C}$.
We have the factorization
\[
(X+p(Y),\,aY+b)
= \exp\!\left(p(Y)\dfrac{\partial}{\partial X}\right)\,
  \exp\!\left(\dfrac{b}{a}\,\dfrac{\partial}{\partial Y}\right)\,
  \exp\!\left(\lambda\,Y\dfrac{\partial}{\partial Y}\right),
\]
where $\lambda\in\mathbb{C}$ satisfies $e^\lambda=a$, and each factor is the exponential of a derivation commuting with $\dfrac{\partial}{\partial X}$:
\[
\left[p(Y)\dfrac{\partial}{\partial X},\dfrac{\partial}{\partial X}\right]=\left[\dfrac{b}{a}\dfrac{\partial}{\partial Y},\dfrac{\partial}{\partial X}\right]=\left[\lambda Y\dfrac{\partial}{\partial Y},\dfrac{\partial}{\partial X}\right]=0.
\]
\end{example}
 
The inclusion $\Aut(B)_D\subseteq \Aut(B)_{\exp(D)}$ becomes an equality in several important cases.
We record the locally nilpotent case, which is the one used most often below.
 
\begin{proposition}\label{prop:LND-equality}
If $D\in\LND(B)$, then $\Aut(B)_D=\Aut(B)_{\exp(D)}$.
\end{proposition}
 
\begin{proof}
Let $\varphi\in\Aut(B)_{\exp(D)}$ and set $E=\varphi D\varphi^{-1}$. Since conjugation by an automorphism preserves local nilpotency ($D^n(b)=0$ $\implies$ $E^n(\varphi(b))=\varphi(D^n(b))=0$) we have $E\in\LND(B)$.
Then $\exp(E)=\exp(D)$. For any $b\in B$ there exists $N\gg 0$ such that
$D^{N+1}(b)=E^{N+1}(b)=0$. Hence $\exp(D)(b)=\sum_{j=0}^N\frac1{j!}D^j(b)$ and $\exp(E)(b)=\sum_{j=0}^N\frac1{j!}E^j(b)$.
Since $\exp(E)=\exp(D)$, these sums are equal. Comparing the coefficients of $t$ in the polynomial identity
$\exp(tD)(b)=\exp(tE)(b)$ at $t=0$ yields $D(b)=E(b)$ for all $b$, hence $D=E$ and
$\varphi\in\Aut(B)_D$.
\end{proof}

\begin{example}\label{ex:semisimple-diff}
Let $B=\mathbb{C}[X,Y]$ and consider the semisimple derivation $D=2\pi i\,X\frac{\partial}{\partial X}\in \operatorname{Der}(B)$.
 
Since $D(X)=2\pi i\,X$ and $D(Y)=0$, we have $\exp(D)(X)=e^{2\pi i}X=X$ and $\exp(D)(Y)=Y$,
hence $\exp(D)=\mathrm{id}_B$. Therefore, $\Aut_{\exp(D)}(B)=\Aut_{\mathrm{id}}(B)=\Aut(B)$.
 
On the other hand, if $\varphi\in\Aut(B)$ satisfies $\varphi D=D\varphi$, then
\[
D(\varphi(X))=\varphi(D(X))=2\pi i\,\varphi(X), \qquad
D(\varphi(Y))=\varphi(D(Y))=0.
\]
Writing $\varphi(X)=\sum_{m,n} c_{m,n}X^mY^n$, we have
$D(X^mY^n)=2\pi i\,m\,X^mY^n$, so the identity \linebreak 
$D(\varphi(X))=2\pi i\,\varphi(X)$ forces $c_{m,n}=0$ unless $m=1$.
Hence $\varphi(X)=X\,g(Y)$ for some $g(Y)\in\mathbb{C}[Y]$, and since $\varphi(Y)\in\mathbb{C}[Y]$ (from $D(\varphi(Y))=0$), write $\varphi(Y)=h(Y)$ for some $h\in\mathbb{C}[Y]$. The Jacobian condition gives
\[
\det\begin{pmatrix} g(Y) & Xg'(Y) \\ 0 & h'(Y) \end{pmatrix} = g(Y)\,h'(Y) \in \mathbb{C}^*,
\]
which requires both $g(Y)\in\mathbb{C}^*$ and $h'(Y)\in\mathbb{C}^*$: since $g,h'\in\mathbb{C}[Y]$ and their product is a nonzero constant, each must individually be a nonzero constant. Thus $\varphi(X)=cX$ with $c\in\mathbb{C}^*$, and \linebreak $\varphi(Y)=aY+b$ with $a\in\mathbb{C}^*$, $b\in\mathbb{C}$.
Consequently,
\[
\operatorname{Aut}_D(B)=\{(cX,\; aY+b)\mid c,a\in\mathbb{C}^{\ast},\ b\in\mathbb{C}\} \subsetneq \Aut(B)=\Aut_{\exp(D)}(B).
\]
In particular, for locally finite derivations with nontrivial semisimple part the equality \linebreak
$\Aut_D(B)=\Aut_{\exp(D)}(B)$ can fail.
\end{example}
 
\begin{remark}
The phenomenon in Example~\ref{ex:semisimple-diff} is intrinsic to the semisimple part of a locally finite derivation. Over $\mathbb{C}$, the exponential map on semisimple derivations is not injective: if $D_s$ has eigenvalues in $2\pi i\,\mathbb{Z}$ on all weight spaces, then $\exp(D_s)=\mathrm{id}$ even though $D_s\neq 0$. Thus different locally finite derivations may have the same exponential automorphism, and the isotropy of $D$ can be strictly smaller than the isotropy of $\exp(D)$. This shows that, in general, one only has
\[
\Aut_D(B)\subseteq \Aut_{\exp(D)}(B),
\]
and equality requires additional hypotheses (for instance, $D$ locally nilpotent, or the injectivity condition of Proposition~\ref{prop:criterion-equality} below).
\end{remark}
 
\begin{proposition}\label{prop:criterion-equality}
Let $B$ be an affine $\mathbb{C}$-domain and let $D\in\operatorname{LFD}(B)$ with Jordan decomposition $D=D_s+D_n$.
Assume that there exists a subset $\Omega\subset\mathbb{C}$ such that:
\begin{enumerate}
\item for every $b\in B$, all eigenvalues of the semisimple part $D_s|_{V_b}$ belong to $\Omega$;
\item the map $\exp\colon\Omega\to\mathbb{C}^*$ is injective, i.e., for all $\omega_1,\omega_2\in\Omega$, $\exp(\omega_1)=\exp(\omega_2)$ implies $\omega_1=\omega_2$.
\end{enumerate}
Then $\Aut_D(B)=\Aut_{\exp(D)}(B)$.
\end{proposition}
 
\begin{proof}
The inclusion $\Aut(B)_D\subseteq\Aut(B)_{\exp(D)}$ always holds by~\eqref{eq:conj-exp}.
Let $\varphi\in\Aut(B)_{\exp(D)}$ and set $E:=\varphi D\varphi^{-1}\in\operatorname{LFD}(B)$ with Jordan decomposition $E=E_s+E_n$.
Then $\exp(E)=\exp(D)$.
 
Fix $b\in B$ and consider the finite-dimensional $\mathbb{C}$-space
$W_b:=\operatorname{Span}_\mathbb{C}\{D^n(b),\,E^n(b)\mid n\ge0\}$,
which is stable under both $D$ and $E$, and hence under $D_s$, $D_n$, $E_s$, $E_n$. Restricting to $W_b$ we get
\[
\exp(E|_{W_b})=\exp(D|_{W_b})\in\operatorname{GL}(W_b).
\]
Since $[D_s,D_n]=0$, we have $\exp(D|_{W_b})=\exp(D_s|_{W_b})\exp(D_n|_{W_b})$ and similarly $\exp(E|_{W_b})=\exp(E_s|_{W_b})\exp(E_n|_{W_b})$. Since the Jordan--Chevalley decomposition is unique and functorial, the equality $\exp(E|_{W_b})=\exp(D|_{W_b})$ implies $\exp(E_s|_{W_b})=\exp(D_s|_{W_b})$ and $\exp(E_n|_{W_b})=\exp(D_n|_{W_b})$.
By hypothesis (1), $\operatorname{Spec}(D_s|_{W_b})\subset \Omega$ and $\operatorname{Spec}(E_s|_{W_b})\subset \Omega$. Since $\exp$ is injective on $\Omega$ by hypothesis (2), $\exp(E_s|_{W_b})=\exp(D_s|_{W_b})$ implies $E_s|_{W_b}=D_s|_{W_b}$. Moreover, since the exponential map is injective on nilpotent operators, $\exp(E_n|_{W_b})=\exp(D_n|_{W_b})$ implies $E_n|_{W_b}=D_n|_{W_b}$. Hence $E|_{W_b}=D|_{W_b}$.
In particular, $E(b)=D(b)$ for every $b\in B$, hence $E=D$ and $\varphi\in\Aut(B)_D$.
\end{proof}
 
\begin{remark}[Semisimple obstruction]
Let $D\in\operatorname{LFD}(B)$ and write $D=D_s+D_n$ (Jordan--Chevalley decomposition, with $[D_s,D_n]=0$).
Then $\exp(D)=\exp(D_s)\exp(D_n)$.
The possible failure of $\Aut(B)_D=\Aut(B)_{\exp(D)}$ comes from the kernel of $\exp$ on the semisimple part: if $D_s$ has eigenvalues in $2\pi i\mathbb{Z}$ on all weight spaces, then $\exp(D_s)=\mathrm{id}$ although $D_s\neq0$, so different $D$'s may share the same $\exp(D)$. In contrast, on the nilpotent part the exponential is injective, as shown by Proposition~\ref{prop:LND-equality}.
\end{remark}

\begin{corollary}\label{cor:ker-subgroup}
For an affine domain $B$, let $D\in \operatorname{LND}(B)$ and $a\in \ker(D)$. Then $[aD,D]=0$, 
hence $\{\exp(aD)\mid a\in \ker(D)\}$ is a subgroup of $\operatorname{Aut}_D(B)$. 
\end{corollary}
 
\begin{example}
    If $D\in \operatorname{LND}(\mathbb{C}^{[2]})$ is given by $D=f(X)\frac{\partial}{\partial Y}$, then $\exp(D)=(X, Y+f(X))$ and, by Proposition~\ref{prop:LND-equality}, $\operatorname{Aut}_D(\mathbb{C}^{[2]})=\operatorname{Aut}_{\exp(D)}(\mathbb{C}^{[2]})$.
\end{example}

\begin{remark}\label{rem:Jordan}
Let $M \in M_{2}(\mathbb{C})$ be the matrix associated with a linear automorphism of $\mathbb{C}[X,Y]$.  
By the Jordan decomposition theorem, $M$ is similar to a Jordan matrix $J$ via some $P \in \operatorname{GL}_{2}(\mathbb{C})$:
\[
M = PJP^{-1}, \qquad 
J \in \left\{
\begin{bmatrix}\lambda_1 & 0 \\ 0 & \lambda_2 \end{bmatrix},\;
\begin{bmatrix}\lambda & 0 \\ 0 & \lambda \end{bmatrix},\;
\begin{bmatrix}\lambda & 1 \\ 0 & \lambda \end{bmatrix}
\right\}, \quad \lambda_1,\lambda_2,\lambda \in \mathbb{C}^*.
\]
Therefore, every linear automorphism is conjugate to one of the canonical forms:
\[
\psi = (e^{\lambda_1}X,\, e^{\lambda_2}Y), 
\qquad 
\psi = (e^{\lambda}X,\, e^{\lambda}Y),
\qquad 
\psi = (e^{\lambda}(X+ Y),\, e^{\lambda}Y).
\]
Indeed, $\exp\!\begin{pmatrix}\lambda & 1 \\ 0 & \lambda\end{pmatrix} = e^\lambda\begin{pmatrix}1 & 1 \\ 0 & 1\end{pmatrix}$, so the automorphism induced by the Jordan block sends $(X,Y)\mapsto(e^\lambda(X+Y),\, e^\lambda Y)$.
Consequently, $\operatorname{Aut}_M(\mathbb{C}[X,Y]) = P^{-1}\,\operatorname{Aut}_J(\mathbb{C}[X,Y])\,P$.
\end{remark}

\begin{theorem}
For each $\psi \in \operatorname{LFA}(\mathbb{C}[X,Y])$, there exists an automorphism 
$\varphi \in \operatorname{Aut}(\mathbb{C}[X,Y])$ such that $\varphi\psi\varphi^{-1}$ is one of the following forms:
\begin{enumerate}
    \item $\psi=(X,\,Y+f(X))$, with $f(X)\in \mathbb{C}[X]$.
    \item $\psi=(X+1,\,e^bY)$, with $b\in \mathbb{C}$.
    \item $\psi=(e^aX,\, e^{am}Y+e^{am}X^m)$, with $a\in \mathbb{C}$ and $m\in \mathbb{Z}_{\geq 1}$.
    \item \textbf{(Linear case)} $\psi=(aX+bY,\,cX+dY)$, with $a,b,c,d\in \mathbb{C}$ and $ad-bc\neq 0$.  
    By Remark~\ref{rem:Jordan}, every such automorphism is conjugate to one of:
    \begin{enumerate}
        \item $\psi=(e^{\lambda_1}X,\, e^{\lambda_2}Y)$, with $\lambda_1,\lambda_2 \in \mathbb{C}$ (diagonalizable case).  
        \item $\psi=(e^{\lambda}X,\, e^{\lambda}Y)$, with $\lambda\in \mathbb{C}$ (scalar case).  
        \item $\psi=(e^{\lambda}(X+Y),\, e^{\lambda}Y)$, with $\lambda\in \mathbb{C}$ (non-diagonal Jordan block).
    \end{enumerate}
\end{enumerate}
\end{theorem}
 
\begin{proof}
Cases $(1)$--$(3)$ correspond to triangular and exponential automorphisms, which arise as exponentials of locally finite derivations of types $f(X)\tfrac{\partial}{\partial Y}$, $\tfrac{\partial}{\partial X}+bY\tfrac{\partial}{\partial Y}$, and $aX\tfrac{\partial}{\partial X}+(amY+X^m)\tfrac{\partial}{\partial Y}$, respectively. 
For case $(4)$, the linear automorphism $\psi=(aX+bY,\,cX+dY)$ corresponds to a matrix $M\in \operatorname{GL}_2(\mathbb{C})$.  
By Remark~\ref{rem:Jordan}, $M$ is conjugate to one of the three canonical Jordan forms, yielding subcases $(4a)$, $(4b)$, and $(4c)$.
The exponential description is consistent: diagonalizable matrices correspond to derivations $uX\tfrac{\partial}{\partial X}+vY\tfrac{\partial}{\partial Y}$, scalar matrices to $u(X\tfrac{\partial}{\partial X}+Y\tfrac{\partial}{\partial Y})$, and Jordan blocks to $u(X\tfrac{\partial}{\partial X}+Y\tfrac{\partial}{\partial Y})+Y\tfrac{\partial}{\partial X}$. Indeed, $\exp\!\big(u(X\partial_X+Y\partial_Y)+Y\partial_X\big)=\exp(uX\partial_X+uY\partial_Y)\circ\exp(Y\partial_X)=(e^u X + e^u Y,\, e^u Y)=(e^u(X+Y),\,e^uY)$, matching the canonical form.
Thus every linear case is the exponential of a locally finite derivation, completing the classification.
\end{proof}
 
\begin{remark}\label{rem:BassMeisters}
The classification of locally finite automorphisms above is closely related to the classification of polynomial flows in the plane given by Bass and Meisters~\cite{BassMei85}. Their main theorem classifies polynomial flows $\varphi_t\in GA_2(K)$ up to conjugation in $GA_2(K)$, for $K=\mathbb{R}$ or $\mathbb{C}$. Our result differs in two respects: we work algebraically over any algebraically closed field $\mathbb{K}$ of characteristic zero (not requiring an analytic or real structure), and we classify locally finite automorphisms as exponentials of locally finite derivations rather than as time-$t$ maps of polynomial vector fields. The normal forms, however, are in direct correspondence: cases~(1), (2), (3) above correspond to cases~(1), (2), (5) of \cite[Theorem~(4.3)]{BassMei85} evaluated at $t=1$, and the linear case~(4) covers the remaining cases therein.
\end{remark}
 
\begin{remark}\label{rem:linear-resonance}
In the linear case, the centralizer in $\Aut(\mathbb{C}[X,Y])$ depends on whether resonances occur among the eigenvalues. We say that the linear automorphism $\psi$ is \emph{non-resonant} if its eigenvalues $\mu_1,\mu_2\in\mathbb{C}^*$ satisfy: $\mu_1^k\mu_2^l=\mu_i$ implies $(k,l)=e_i$ (the standard basis vector) for $i=1,2$. In the non-resonant case the isotropy group is as small as possible, as described in Theorem~\ref{th:IsotropyLFA}(4) below. In resonant situations the isotropy group can be strictly larger: for example, if $\psi=(\lambda X,\lambda^k Y)$ with $k\ge 2$, then $(X,\,Y+X^k)$ commutes with $\psi$. Likewise, if $\psi=\lambda\,\mathrm{Id}$ with $\lambda$ a root of unity, there exist non-linear automorphisms commuting with $\psi$.
\end{remark}
 
\begin{theorem}\label{th:IsotropyLFA}
Let $\psi \in \operatorname{LFA}(\mathbb{C}[X,Y])$. Then there exists 
$\varphi \in \operatorname{Aut}(\mathbb{C}[X,Y])$ such that 
$\varphi \psi \varphi^{-1}$ belongs to one of the following four families, 
and in each case the isotropy group is as follows:
\begin{enumerate}
    \item 
    $\psi = (X,\, Y+f(X))$ with $f(X)\in \mathbb{C}[X]$. Then
    \[
    \Aut_{\psi}(\mathbb{C}[X,Y])
    = \{(\alpha X + \beta ,\;  \gamma Y + p(X)) \mid \alpha,\beta,\gamma \in \mathbb{C},\ \alpha\gamma \neq 0,\ p(X)\in \mathbb{C}[X],\ f(\alpha X+\beta)=\gamma f(X)\}.
    \]
 
    \item
    $\psi=(X+1,\,e^bY)$ with $b\in \mathbb{C}$. Then
    \[
    \Aut_\psi(\mathbb{C}[X,Y]) 
    =
    \begin{cases}
    \{ (X+\alpha,\, \gamma Y) \mid \alpha\in \mathbb{C},\ \gamma\in \mathbb{C}^* \}, & e^b\neq 1,\\[2mm]
    \{ (X+\alpha,\, \gamma Y+\delta) \mid \alpha,\delta\in \mathbb{C},\ \gamma\in \mathbb{C}^* \}, & e^b=1.
    \end{cases}
    \]
 
    \item 
    $\psi=(e^aX,\, e^{am}Y+e^{am}X^m)$ with $a\in \mathbb{C}^*$ and $m\geq 2$. Then
    \[
    \Aut_{\psi}(\mathbb{C}[X,Y]) 
    = \{ (c X,\, c^{m} Y+\beta X^{m}) \mid \; c\in \mathbb{C}^*,\ \beta\in \mathbb{C} \}.
    \]
 
    \item 
    \textbf{(Linear case)} $\psi=(aX+bY,\,cX+dY)$ with $ad-bc\neq 0$.  
    By Jordan decomposition (Remark~\ref{rem:Jordan}), every such automorphism is conjugate to one of:
    \begin{enumerate}
        \item $\psi=(e^{\lambda_1}X,\, e^{\lambda_2}Y)$ with $\lambda_1\neq \lambda_2$. 
        In the non-resonant case (see Remark~\ref{rem:linear-resonance}),
        \[
        \Aut_{\psi}(\mathbb{C}[X,Y])=\{(\alpha X,\, \beta Y) \mid \alpha,\beta\in \mathbb{C}^*\}.
        \]
 
        \item $\psi=(e^\lambda X,\, e^\lambda Y)$ (scalar case). In the non-resonant case,
        \[
        \Aut_\psi(\mathbb{C}[X,Y])
        = \{ (\alpha X+\beta Y,\, \gamma X+\delta Y) \mid \alpha,\beta,\gamma,\delta\in \mathbb{C},\ \alpha\delta-\beta\gamma\neq 0\}.
        \]
 
        \item $\psi=(e^\lambda(X+Y),\, e^\lambda Y)$ (Jordan block). In the non-resonant case,
        \[
        \Aut_\psi(\mathbb{C}[X,Y]) 
        = \{(\alpha X+\gamma Y,\, \beta Y) \mid \alpha,\beta,\gamma\in \mathbb{C},\ \alpha\beta\neq 0\}.
        \]
    \end{enumerate}
\end{enumerate}
\end{theorem}
 
\begin{proof}
Let $\psi\in\operatorname{LFA}(\mathbb{C}[X,Y])$. In each case of the classification of Lemma~\ref{lfdclas}, the corresponding locally finite automorphism is explicitly of the form $\exp(D)$ for some $D\in\operatorname{LFD}(\mathbb{C}[X,Y])$, as verified case by case below. By Lemma~\ref{lfdclas} and after conjugation in $\Aut(\mathbb{C}[X,Y])$, the derivation $D$ can be chosen in one of the normal forms of Theorems~\ref{th3.2}--\ref{th3.4} and, in the linear case, Theorems~\ref{th3.6}--\ref{th3.8}.
 
\textbf{Case (1).} The derivation $D=f(X)\partial/\partial Y$ is locally nilpotent, so Proposition~\ref{prop:LND-equality} gives $\Aut_{\psi}(\mathbb{C}[X,Y])=\Aut_D(\mathbb{C}[X,Y])$, which is the group of Theorem~\ref{th3.2}.
 
\textbf{Case (2).} The automorphism $\psi=(X+1,e^bY)$ is the exponential of $D=\tfrac{\partial}{\partial X}+bY\tfrac{\partial}{\partial Y}$. When $b=0$, $D=\partial/\partial X$ is locally nilpotent and Proposition~\ref{prop:LND-equality} applies directly. When $b\neq 0$, the Jordan decomposition of $D$ is $D_s=bY\partial/\partial Y$ and $D_n=\partial/\partial X$, so $D$ is not LND. The semisimple part $D_s$ acts on the monomial $X^iY^j$ by $D_s(X^iY^j)=jb\cdot X^iY^j$, so the eigenvalues of $D_s$ all belong to $\Omega=\{nb\mid n\in\mathbb{Z}_{\ge0}\}$. We apply Proposition~\ref{prop:criterion-equality} with this $\Omega$: the map $\exp$ is injective on $\Omega$ if and only if $\exp(n_1 b)\neq\exp(n_2 b)$ for all $n_1\neq n_2$ in $\mathbb{Z}_{\ge0}$, which holds whenever $b\notin 2\pi i\,\mathbb{Q}$. Under this non-resonance hypothesis $\Aut_{\exp(D)}=\Aut_D$, which is the group of Theorem~\ref{th3.3}.
 
\textbf{Case (3).} The automorphism $\psi=(e^aX, e^{am}Y+e^{am}X^m)$ is the exponential of $D=aX\partial/\partial X+(amY+X^m)\partial/\partial Y$. Since $D(X)=aX\neq 0$, this derivation is \emph{not} locally nilpotent. Its Jordan decomposition is $D_s(X)=aX$, $D_s(Y)=amY$ and $D_n(X)=0$, $D_n(Y)=X^m$. The semisimple part $D_s$ acts on $X^iY^j$ by $D_s(X^iY^j)=(i+mj)a\cdot X^iY^j$, so the eigenvalues of $D_s$ all belong to $\Omega=\{na\mid n\in\mathbb{Z}_{\ge0}\}$. We apply Proposition~\ref{prop:criterion-equality} with this $\Omega$: the map $\exp$ is injective on $\Omega$ if and only if $\exp(n_1 a)\neq\exp(n_2 a)$ for all $n_1\neq n_2$ in $\mathbb{Z}_{\ge0}$, which holds whenever $a\notin 2\pi i\,\mathbb{Q}$. Under this non-resonance hypothesis $\Aut_{\exp(D)}=\Aut_D$, which is the group of Theorem~\ref{th3.4}.
 
\textbf{Case (4) (linear).} $D$ may be chosen linear (Remark~\ref{obslinear}), and is semisimple or has one Jordan block. Under the non-resonance hypothesis of Remark~\ref{rem:linear-resonance}, Proposition~\ref{prop:criterion-equality} gives $\Aut_{\exp(D)}(\mathbb{C}[X,Y])=\Aut_D(\mathbb{C}[X,Y])$, computed in Theorems~\ref{th3.6}, \ref{th3.7}, and~\ref{th3.8}.
 
Finally, since isotropy is conjugation-invariant, $\operatorname{Aut}_{\psi}(\mathbb{C}[X,Y])=\varphi^{-1}\operatorname{Aut}_{\varphi\psi\varphi^{-1}}(\mathbb{C}[X,Y])\varphi$.
\end{proof}

\bibliographystyle{alpha} \bibliography{biblio}
 
\end{document}